\documentclass{amsart}
\usepackage[utf8]{inputenc}
\usepackage[english]{babel}
\usepackage{latexsym}
\usepackage{amsmath,amssymb,amsfonts,amsthm}
\usepackage{mathrsfs}
\usepackage{color}
\usepackage{caption}
\usepackage[all]{xy}

\title[Remarks on foliations with a unique singular point]{Remarks on foliations on $\CP$ with a unique singular point}
\author{Claudia R. Alcántara}
\address[Claudia R. Alcántara]{Departamento de Matemáticas, Universidad de Guanajuato, Callejón Jalisco s/n, A.P. 402, C.P. 36000, Guanajuato, Gto., México}
\email{claudia@cimat.mx}

\thanks{First author partially supported by GIR ECSING of the University of Valladolid and CONACyT under Grant 284424.}
\author{Jorge Mozo-Fern\'{a}ndez}
\address[Jorge Mozo Fern\'{a}ndez]{Dpto. \'{A}lgebra, An\'{a}lisis Matem\'{a}tico, Geometr\'{\i}a y Topolog\'{\i}a \\
Facultad de Ciencias, Universidad de Valladolid \\
Campus Miguel Delibes\\
Paseo de Bel\'{e}n, 7\\
47011 Valladolid - Spain}
\email{jorge.mozo@uva.es}

\thanks{Second author partially supported by Ministerio de Ciencia e Innovación of Spain, under Project PID2019-105621GB-I00, of title \textit{Métodos asintóticos, algebraicos y geométricos en foliaciones singulares y sistemas dinámicos} (Javier Sanz Gil and Fernando Sanz Sánchez, coords.), by Ministerio de Ciencia e Innovación, under Project PID2022-139631NB-I00, \textit{Análisis Asintótico, Álgebra y Geometría en Sistemas Dinámicos} (Jorge Mozo Fernández and Fernando Sanz Sánchez, coords.), and by GIR ECSING of the University of Valladolid.}

 \newtheorem{theorem}{Theorem}[section]
\newtheorem{lem}[theorem]{Lemma}
\newtheorem{example}[theorem]{Example}
\newtheorem{defin}[theorem]{Definition}
\newtheorem{prop}[theorem]{Proposition}
\newtheorem{cor}[theorem]{Corollary}
\newtheorem{remark}[theorem]{Remark}

\newcommand{\Oc}{\mathcal{O}}

\newcommand{\CP}{\mathbb{CP}^2}
\newcommand{\C}{\mathbb{C}}

\newcommand{\Q}{\mathbb{Q}}

\newcommand{\N}{\mathbb{N}}
\newcommand{\F}{\mathcal{F}}

\newcommand{\pa}[1]{\frac{\partial}{\partial #1}}
\newcommand{\XX}{{\mathcal X}}

\newcommand{\norm}[1]{\left | #1 \right |} 

\DeclareMathOperator{\ord}{ord}
\DeclareMathOperator{\Sing}{Sing}

\begin{document}

\maketitle

\begin{abstract}
This paper is a contribution to the study of foliations on $\CP$ with a unique singularity. We provide an explicit example in degree 7 of such a foliation, in the non dicritical case, having a divergent separatrix, and subsequently, we study different families of such foliations in low degree. 
\end{abstract}

\section{Introduction}
In this paper, we are interested in the study of codimension one   holomorphic foliations over the complex projective plane $\CP$. 
Roughly speaking, a codimension one holomorphic foliation $F$ over a (complex) analytic or algebraic surface $M$ is a partition of $M$ in one dimensional Riemann surfaces (holomorphic curves) varying analytically: the leaves of the foliation. Foliations may be defined locally either by holomorphic 1-forms or by vector fields. In Section \ref{preliminares} we will precise this notion adapted to the projective plane, where a foliation may be defined by a homogeneous polynomial 1-form.

The 1-forms defining a foliation may vanish at a discrete number of points, called singularities of the foliation. For instance, foliations over $\CP$ always have sngularities, its number being bounded in terms of the degree of the foliation, which is a consequence of Proposition \ref{milnor}. The study of foliations may be carried out from two complementary points of view: locally and globally. Locally, people is interested, for instance, in knowing the properties of the foliation and its leaves in a neighbourhood around the singular points, and also, to state when two such foliations are equivalent from different points of view. One of the main ingredients used in the local study of codimension one holomorphic foliations over an ambient space of dimension two is the reduction of singularities, which can be found in \cite{vandenessen,mattei-moussu,mozo,cano-cerveau-deserti} and will be briefly recalled in Section \ref{preliminares}.

If approaching the study globally, people is interested in other problems such as the existence of leaves having particular properties, the sets of accumulation of these leaves, or the degree of the algebraic leaves (provided that they exist), when $M$ is an algebraic variety. 

Assume, in the sequel, that $F$ is a degree $d$ foliation (degree will be  defined in Section \ref{preliminares}) over the projective plane $\CP$. Some relevant problems concerning projective foliations are:
\renewcommand\theenumi{\arabic{enumi}}
\renewcommand\labelenumi{\theenumi.}
\begin{enumerate}
\item Existence (or not) of invariant algebraic curves.  It was J.-P. Jouanolou who gave in \cite{jouanolou} the first example of such a foliation without invariant algebraic curves. It is a foliation of degree $d$ with maximum number of singular points, and plenty of symmetries. Since then, other examples of such foliations have been given, and it is known that the set of foliations of degree $d$ without invariant algebraic curves contains a dense open set of the space of all foliations.
\item Bounding the degree of algebraic invariant curves, which has received attention from many authors. M. Carnicer shows in \cite{carnicer2} that, if all singularities are non dicritical, the maximum degree of an invariant curve of a foliation of degree $d$ is $d+2$. Another concise proof, using indices, can be read in \cite{brunella}. In the dicritical case, no such bound in terms of the degree exists.

\item The minimal set problem asks if there exists a foliation $F$ on $\CP$ of degree $d \geq 2$ with a non-trivial minimal set. This problem is equivalent to find a non-singular leaf that does not accumulate on any singular point.

\end{enumerate}

Foliations on $\CP$ have always a certain number of singularities, bounded   above by $d^2+d+1$.  A rather interesting problem is to study foliations which are  in the opposite situation as the one of Jouanolou's example, i.e.,  having only one singular point, of maximal Milnor number $d^2+d+1$ (the notion of Milnor number will be recalled in Section \ref{preliminares}). These foliations have been studied by C.R. Alcántara and P.R. Pantaleón in \cite{alcantara,alcantara-pantaleon},  and, in this Bulletin, by S.C. Coutinho and F.R. Ferreira \cite{CF,CF2}. In particular, foliations having one singular point with nonzero linear part are studied: in \cite{alcantara-pantaleon} it is shown that this singular point must be, either a saddle-node or a nilpotent singularity of the type called, following E. Str\'{o}\.{z}yna and H. \.{Z}o{\l}{\c{a}}dek \cite{SZ}, generalized saddle-nodes. In either case, no algebraic invariant curves appear.  Coutinho and Ferreira exhibit a family of foliations of degree $k+1$ having only one singular point of multiplicity $k$. Even if foliations without algebraic leaves seem to be the generic case, few explicit examples are known, and foliations with one singular point verifying particular conditions provide a source of such examples.

In this work, our aim is to deepen the study of foliations with only one singular point, establishing local properties of the singularity and the separatrices through it, coming from the global structure of the foliation. One of our objectives would be to study when a projective foliation with a unique singularity can have a divergent separatrix through it. We shall review the cases where the linear part of the foliation at the singular point is non zero, and subsequently, we shall study other cases. 

The structure of the paper is as follows: in Section 2 the main definitions about foliations, and properties are reviewed. In Section 3 
we describe how the Cremona transformations justify the interest of studying foliations with only one singular point, and will show that, in fact, in this case divergent separatrices may appear. We will provide an explicit example, of such a foliation, which is not a second type one. In fact, we will show that this kind of foliations, second type ones, form a family invariant by Cremona transformations. The case of unique singularity with non-zero linear part is studied in Section \ref{unica-singularidad}, providing in particular alternative proofs of the non algebraicity of the separatrices. These are foliations of second type, but when the degree is $d=2$, all foliations with one singular point are classified, and some of them are not of second type (while they don't exhibit divergent separatrices). Foliations with one singularity which are, either generalized curves or second type foliations, are studied in Section 5, and in Section 6 we will consider the case of foliations with rational first integral. Finally, in Section 7, some foliations of degree three are considered (unstable foliations, as presented in \cite{alcantara-ronzon}), showing that their separatrices are, again, convergent.

\section{Preliminaries} \label{preliminares}

The main object we are going to work with in this paper is holomorphic \mbox{foliations} of codimension 1 on $\CP$. This section will be devoted to recall the main definitions and properties we will use throughout the text. More details may be found in \cite{brunella,cano-cerveau-deserti,lins-scardua,gmont}. Homogeneous coordinates will be denoted $(x:y:z)$. Such a foliation $F$ may be described in two equivalent ways:
\begin{enumerate}
\item By a homogeneous vector field
$$
\XX= P(x,y,z)\pa{x}+ Q(x,y,x)\pa{y} + R(x,y,z)\pa{z},
$$
where $P$, $Q$, $R\in \C [x,y,z]$ are homogeneous polynomials of the same degree $d$, without common factors. This vector field describing the foliation is not unique: if $G(x,y,z)$ is a homogeneous polynomial of degree $d-1$ and $\mathcal{R}=x\pa{z}+y\pa{y}+z\pa{z}$ is the radial vector field, both $\XX$ and $\XX + G\cdot \mathcal{R}$ describe the same foliation and conversely, two vector fields representing the same foliation differ in a multiple of $\mathcal{R}$ up to multiplication by a nonzero scalar.

\item By a homogeneous 1-form $\omega=A(x,y,z)dx+ B(x,y,z) dy+ C(x,y,x) dz$, with $A$, $B$ and $C$ homogeneous of degree $d+1$ satisfying Euler's condition $xA+yB+zC=0$.
\end{enumerate}

Both expressions for a holomorphic foliation are related as follows. If $F$ is \mbox{described} by a vector field $\XX$ as above, the 1-form defining the foliation can be written as
$$
\begin{vmatrix}
dx & dy & dz \\ x & y & z \\ P & Q & R 
\end{vmatrix}= (yR-zQ)dx+(zP-xR)dy+(xQ-yP)dz.
$$
In this text, we will use the notation $F$ for an abstract foliation, $\XX$ for its corresponding meromorphic vector field, and $\omega$ for the corresponding 1-form. 
 \\
 
 Using previous considerations we see that the space of foliations of degree $d$ on $\CP$ is a projective space of dimension $d^2+4d+2$. Several notions, both local and global, about foliations, will be needed in this work, and will be recalled in this section. Given a holomorphic foliation $F$ on $\CP$ and an algebraic curve $\mathcal{C}$ defined by a homogeneous polynomial $f$, we will say that $\mathcal{C}$ is invariant if it is tangent to the foliation at every regular point. Algebraically, given a 1-form $\omega$ defining the foliation, invariance of $\mathcal{C}$ means that a 2-form exists such that $\omega \wedge df=f \eta$.

Given a non-invariant line $L$, the degree $d$ of the foliation may be computed as the number of tangencies (counted with multiplicities) of $L$ with the foliation. A point $p_0=(x_0: y_0: z_0)\in \CP$ is singular for $F$ if all coefficients (either of the vector field or the 1-form defining $F$) vanish at $p_0$. Assume that $p_0=(1:b :c )$ and restrict the foliation to the local chart $x=1$: $p(y,z)dy+q(y,z)dz$. The \textbf{Milnor number} of $F$ at $p_0$ is defined as $\mu_{p_0}(F):=\dim_{\C}\dfrac{\Oc_{\C^2,(b,c)}}{<p,q>}$ (this dimension is finite, as singularities are isolated), where, as usual, $\Oc_{\C^2,(b,c)}$ denotes the local ring of germs at $(b,c)$, 
and  \textbf{the multiplicity of $p_0$} as $m_{p_0}(F):=\min\{ \nu_{(b,c)}(p),\nu_{(b,c)}(q)\}$, where $\nu$ denotes the order at $(b,c)$ of a formal power series.  Note that we are considering $p$ and $q$ as elements of $\Oc_{\C^2,(b,c)}$ using an immersion of the polynomial ring into this local ring. These notions may be defined formally, using formal power series, and are invariant under formal changes of coordinates. 

The following well-known result relates to the Milnor number and the degree:

\begin{prop}[\cite{jouanolou}]  \label{milnor} Let $F$ be a foliation on $\CP$ of degree $d$ with isolated singularities. Then
\begin{equation*}
d^2+d+1=\sum_{p \in \CP} \mu_p(F).
\end{equation*}
\end{prop}

Assume now that $p=(1:0:0)$ is a singular point of multiplicity 1, and let $A$ be the matrix of the linear part of a vector field defining locally the foliation at $p$. Different cases will be under consideration:
\begin{enumerate}
\item If $A$ is nilpotent and non zero, $p$ will be called a \textbf{nilpotent singularity}.
\item If $A$ has a zero eigenvalue and a nonzero one, $p$ will be called a \textbf{saddle-node}.
\end{enumerate}

It is well known that a reduction of singularities exists for holomorphic foliations in dimension two. This means that, after a finite number of point blow-ups, a new foliation is obtained and its singular points are of one of the following two types:
\begin{enumerate}
\item Either saddle-nodes.
\item Or singularities with nonzero eigenvalues $\lambda_1$, $\lambda_2$ satisfying $\lambda_1/\lambda_2 \notin \Q_{>0}$ (hyperbolic singularities).
\end{enumerate}

For details about reduction of singularities, the reader is addressed to \cite{cano-cerveau-deserti, mattei-moussu, mozo}. If $\omega$ is a 1-form defining a foliation locally around a reduced singularity $p$ (which we assume again that it is the origin), there exist at most two smooth analytic curves (separatrices) through $p$. In the hyperbolic case, exactly two such separatrices exist, each of them tangent to a different eigendirection. In the saddle-node case, there always exists a separatrix tangent to the eigendirection of the nonzero eigenvalue (the strong separatrix). It may exist (or not) another one tangent to the eigendirection of the null eigenvalue (the weak separatrix). Anyhow, it always exists formally, while in general its convergence cannot be guaranteed.

Point blow-ups allow us to reduce singularities of foliations, but having as main disadvantage a change of the ambient space: we no longer live on $\CP$ but on a different surface. 

Related to this reduction process we have the following important notions, which will be used throughout the work.

\begin{defin} \cite{camacho-lins-sad} A foliation $F$ on $\CP$ is a \textbf{generalized curve} if no saddle-nodes appear in its reduction process. 
\end{defin}
\begin{defin} \cite{Mattei-Salem} A foliation $F$ on $\CP$ is of \textbf{second type} if, for every saddle-node in its reduction process, the weak separatrix is not contained in any component of the exceptional divisor.
\end{defin}

Note that if a foliation $F$ is a generalized curve then it is of second type.
Associated to a singular point of a foliation defined by the 1-form $\omega$, several indices may be defined which will be used throughout the text. We will give a very quick revision of them.

\begin{enumerate}
\item CS (Camacho-Sad) index. If $p_0$ is a singularity, and $f$ is the (reduced) equation of a separatrix $\mathcal{C}$ through $p_0$, there are relatively prime germs $g,h\in \C \{ y,z\}$ and a 2-form $\eta$ such that $g \omega =hdf+f\eta$. The Camacho-Sad index of $\mathcal{C} $ at $p_0$ is
\begin{equation} \label{CS-definition}
CS_{p_0} (F, \mathcal{C})=-\frac{1}{2\pi i} \int_{\gamma} \frac{\eta}{h} = Res_{p_0} \left( \left. -\frac{\eta}{h} \right|_{\mathcal{C}} \right),
\end{equation}
where $\gamma$ is a union of small circles around $p$, one for each irreducible component of $\mathcal{C}$. It is straightforward to compute it in some significant cases:
\begin{itemize}
\item If $p_0$ is a hyperbolic singularity, write the foliation in local coordinates as  $\omega= \lambda_1 z U_1 (y,z) dy + \lambda_2 y U_2 (y,z) dz$, $U_1 (0)= U_2 (0)=1$. If $\mathcal{C}$ is $z=0$, then $CS_{p_0} ( F , \mathcal{C})= -\dfrac{\lambda_1}{\lambda_2}$.
\item If $\omega$ is a saddle-node with a normal form $\omega_{k,\lambda}= z^{k+1} dy - y (1+\lambda z^k)dz$, $z=0$ is the strong separatrix and $y=0$ is the weak one. We have
\begin{equation*}
CS_{p_0} (F, z=0)=0;\qquad CS_{p_0} (F, y=0)=\lambda.
\end{equation*}
\end{itemize}

\item GSV (Gómez-Mont-Seade-Verjovsky) index. Given a germ of separatrix $\mathcal{C}$ through $p_0$, with reduced equation $f$, and the decomposition $g\omega =h df+f\eta$, GSV index of $\mathcal{C}$ at $p_0$ is defined as
\begin{equation} \label{GSV-definition}
GSV_{p_0} (F,\mathcal{C}) = \frac{1}{2\pi i} \int_\gamma \frac{g}{h} d \left( \frac{h}{g}\right).
\end{equation}

If it is smooth and $\omega=ady+bdz$, assuming that $p_0=(0,0)$, it can be easily checked that
\begin{equation} \label{GSV-smooth}
GSV_{p_0} (F, \mathcal{C}) = \dim_{\C} \frac{\Oc_{\C^2,(0,0)}}{(a,b,f)}.
\end{equation}
In particular, if $\mathcal{C}$ is $z=0$, and $\omega=za_1 dy+bdz$, $GSV_{p_0} (F, \mathcal{C})= \ord_{p_0} (b(y,0))$.
If $\mathcal{C}$ is either any of the separatrices through a hyperbolic singularity $p_0$ or the strong separatrix of a saddle-node, then $GSV_{p_0}(F, \mathcal{C})=1$. If $\mathcal{C}$ is the weak one (provided it exists), then $GSV_{p_0} (F, \mathcal{C})=k+1$.
\end{enumerate}

Both indices can be defined in a purely formal setting. In order to do it, it is enough to take a Puiseux parametrization $\sigma$ of a branch of the separatrix $\mathcal{C}$, and compute
$$
Res_{0}\left( -\sigma^{\ast}\left( -\frac{\eta}{h} \right) \right)
$$
in the case of Camacho-Sad index, or
$$
Res_{0} \left(\sigma^{\ast}\left(  \frac{g}{h} d \left( \frac{h}{g}\right) \right) \right)
$$
for GSV. Note also that
$$
\frac{\Oc_{\C^2,(0,0)}}{(a,b,f)} \cong \frac{\widehat{\Oc}_{\C^2,(0,0)}}{(a,b,f)} \cong \frac{\C [[y,z]]}{(a,b,f)},
$$
so the formula \eqref{GSV-smooth} is also valid formally.

The following result relates the degree of a foliation and the previously defined indices.

\begin{theorem} Let $F$ be a foliation on $\CP$ of degree $d$ with isolated singularities and suppose that $F$ has an invariant algebraic curve $\mathcal{C}$ of degree $t$. Then
\begin{equation} \label{CS}
\sum_{p \in \Sing (F ) \cap \mathcal{C}} CS_p(F,\mathcal{C}) =t^2 \end{equation}
and
\begin{equation} \label{GSV}
\sum_{p \in \Sing (F ) \cap \mathcal{C}} GSV_p(F,\mathcal{C})=(d+2)t-t^2.
\end{equation}
\end{theorem}
\begin{enumerate}
\setcounter{enumi}{2}
\item BB (Baum-Bott) index. Assume that the foliation $F$ is given locally by the vector field $\XX= Q(y,z)\pa{y}+ R(y,z)\pa{z}$. The Baum-Bott index of $F $ at $p_0$ is
\begin{equation*}
BB (F, p_0)= Res_{p_0} \left( \frac{(tr J)^{2}}{Q\cdot R} dy\wedge dz \right),
\end{equation*}
where $J$ is the Jacobian matrix of $(Q,R)$ at $p$. If $p$ is hyperbolic with eigenvalues $\lambda_1$, $\lambda_2$, we have that
\begin{equation*}
BB (F, p_0)= 2+ \frac{\lambda_1}{\lambda_2} + \frac{\lambda_2}{\lambda_1}.
\end{equation*}
For $p$ a saddle-node, $BB(F,p)= 2p+2+\lambda$.

The following formula will be used throughout the text:
\begin{equation} \label{BB}
BB(F):=\sum_{p\in \Sing (F)} BB(F,p)= (d+2)^2,
\end{equation}
where, as before, $d$ denotes the degree of the foliation.

\end{enumerate}

Finally, if $F$ is a second type foliation at the singular point $p$, these three indices are related as follows \cite{fernandez-mol}:
\begin{equation} \label{formula-indices}
    BB(F,p)= CS_p(F,\mathcal{C}_p) + 2GSV_p (F,\mathcal{C}_p),
\end{equation}
where $\mathcal{C}_p$ denotes the set of separatrices (maybe formal) through $p$.

\section{Cremona Transformations}

Foliations over $\CP$ may be transformed using Cremona transformations: birational transformations from $\CP$ to itself. A Cremona transformation  $T$ of $\CP$ is well defined away from its indeterminacy set $\textrm{Ind}(T)\subseteq \CP$ (a finite number of points). The set where the Jacobian determinant of $T$ vanishes is the exceptional set of $T$, $\textrm{Exc} (T)$: it is a finite set of curves that collapse to points. Given a foliation $F$, $T^{\ast} F$ may be extended to define a new foliation on $\CP$. 

The standard Cremona transformation is the map defined by $\sigma (x:y:z)= (yz:xz:xy)$, with indeterminacy set $\{ (1:0:0), (0:1:0), (0:0:1)\}$, and exceptional set $xyz=0$. This map collapses three lines to points, and blows up three points to obtain new lines. In fact, it is classical that $\sigma$ is a composition of three blow-ups and three blow-downs. 

A Cremona transformation that will be useful for us is the map $\rho (x:y:z)= (xy:y^2 : x^2-yz)$, with indeterminacy set $\{ (0:0:1)\}$ and exceptional set $y=0$. If several singularities lie over the line $y=0$, and are different from the point $(0:0:1)$, $\rho$ collapses them to a single point. Thus, applying a sequence of these transformations and automorphisms of $\CP$, any foliation over $\CP$ may transformed into a foliation with only one singularity (see \cite{cerveau-deserti} for details). Many properties are preserved under birational maps (algebraicity of the leaves, divergence of the separatrices, etc), so this result supports the interest in studying foliations with only one singularity. As the main disadvantage, this singularity may have a big multiplicity and the new foliation obtained a big degree, as with every Cremona map, the degree may increase. In this work, we will concentrate on small multiplicities and degrees. 

As a remark, another approach is possible: according to \cite{carnicer}, a kind of reduced singularities may be obtained using Cremona transformations, increasing the number of singular points. We will not follow this approach (complementary to ours) in this work.

The following is an important result about the relationship between the desingularization process and the foliations obtained using Cremona transformations.

\begin{theorem} \label{cremona-second type}
Under a Cremona transformation, a second type foliation (resp. a generalized curve) is transformed into a second type foliation (resp. a generalized curve).
\end{theorem}

\begin{proof}
Even if the argument is general, we will explicit it in the case of the Cremona map $\rho$, as defined before, which is enough for our purposes. The indeterminacy set is eliminated after three blow-ups of points, so we have an algebraic morphism $T=\rho \sigma_{1}\sigma_{2} \sigma_{3}: M\longrightarrow \CP$. These blow-ups generate three components of the exceptional divisor, $E_i$, $i=1,2,3$, with respective self-intersections $-2$, $-2$ and $-1$. Moreover, the strict transform $\tilde{L}$ of the exceptional line $L$ has self-intersection $-1$. 

This map $T$ is a composition of blow-ups and automorphisms (see \cite[Ch. IV.3]{shafarevich}, \cite{deserti} for details). In fact, $T$ can be written as $\tau_{3} \tau_{2} \tau_{1}$, where $\tau_3$ collapses $\tilde{L}$, $\tau_{2}$ collapses $\tau_3 (E_2)$, and $\tau_{1}$ collapses $\tau_2 \tau_1 (E_1)$. The image of $E_3$ is the exceptional set. 

Given a second type singular point, under a blow up the singular points appearing are again second type, by definition. So, if $F$ is a second type foliation, $T^{\ast} F$ also is. Consider now the transformed foliation of $F$ by $\rho$ ($F'=\rho^{\ast} F$). As $(\tau_{3} \tau_{2} \tau_{1})^{\ast} F' = T^{\ast} F$, the only possibility to have a saddle point in a bad position (either in a corner or in a regular point of the exceptional divisor, with weak separatrix contained in the divisor) would be in the first three blow-ups. But, if this happens, this situation doesn't disappear under further blow-ups, and this would lead to a contradiction.

A similar argument applies to generalized curves.
\end{proof}
We are going to use previous tools in order to explicit a foliation on $\CP$ with one single singularity and a divergent separatrix through it. Take 
the projectivization of Euler's foliation $y^2dz +(z-y)dy$, which has the equation $- (y^2 z+xy (z-y))dx + x^2 (z-y)dy + xy^2 dz $, singular points $(1:0:0)$, $(0:1:0)$, $(0:0:1)$, and invariant lines $y=0$, $x=0$. It turns out that, after blowing up $(0:0:1)$ a saddle-node appears with a weak separatrix contained in the divisor, so it is not a second type foliation.
The sequence of transformations to apply can be described as follows:
\begin{enumerate}
\item Make a change of variables $z\rightarrow z-x$, in order that the point $(0:0:1)$ is not singular.
\item The Cremona map $\rho_1 (x:y:z)= (xy:y^2:x^2-yz)$ blows up the point $(0:0:1)$ and collapses the line $y=0$.
\item The line $x=0$ is still invariant. Again, after the change of variables $y\rightarrow y-z$, $(0:0:1)$ is not singular.
\item The map $\rho_2 (x:y:z)$ blows up $(0:0:1)$ and collapses $x=0$.
\end{enumerate}

After this sequence, we get a degree 7 foliation with only one singularity of multiplicity 4 at $(0:0:1)$, which has Milnor number $57$. The analytic expression of this foliation, computed with Maple, is as follows:
\begin{align*}
\tilde{\omega} & = (-x^7y - x^7z + 4x^6y^2 + 4x^6yz + 2x^6 z^2 - 5x^5 y^3 - 4 x^5 y^2 z + x^5 y z^2 + x^4 y^4  \\ & - 4 x^4 y^3 z - 3 x^4 y^2 z^2 - 3 x^4 y z^3 - x^4 z^4 + 4 x^3 y^5 + 9 x^3 y^4 z + 12 x^3 y^3 z^2 + 5 x^3 y^2 z^3 \\ & - 6 x^2 y^6  - 15 x^2 y^5 z - 9 x^2 y^4 z^2 + 6 x y^7 + 7 x y^6 z - 2 y^8)dx \\ &  + (x^8 - 4 x^7 y - 2 x^7 z + 5 x^6 y^2 + 3 x^6 y z - x^5 y^3 + 2 x^5 y^2 z - 4 x^4 y^4 - 6 x^4 y^3 z \\ & - 6 x^4 y^2 z^2 - 2 x^4 y z^3 + 6 x^3 y^5 + 12 x^3 y^4 z + 6 x^3 y^3 z^2 - 6 x^2 y^6 - 6 x^2 y^5 z + 2 x y^7)dy \\ &  + (x^8 - 2 x^7 y - 2 x^7 z + x^6 y^2 - x^6 y z + 2 x^5 y^3 + 3 x^5 y^2 z \\ & + 3 x^5 y z^2 + x^5 z^3 - 3 x^4 y^4 - 6 x^4 y^3 z - 3 x^4 y^2 z^2 + 3 x^3 y^5 + 3 x^3 y^4 z - x^2 y^6)dz
\end{align*}
This foliation is, in fact, one element in a family of foliations with one singular point and a divergent separatrix through it. Note that there are algebraic invariant curves ($x=0$) is invariant and it is not a second type foliation.

This example may be generalized in different ways. Among the simplest ones, consider the saddle-node written in local coordinates as
$$
y^d dz+(Q_p (y,z) + z+ \varepsilon y)dy,
$$
where $d\geq 2$, $Q_d (y,z)$ is a homogeneous polynomial of degree $d$ and $\varepsilon \in  \{ 0,1\}$. Its homogenization is 
$$
x(Q_d (y,z) + x^{d-1} (z+\varepsilon y) )dy +xy^d dz -(yQ_d (y,z) + x^{d-1} y (z+\varepsilon y) + y^d z) dx,
$$
which has singularities at infinity in the points solution of
$$
yQ_d (y,z) + y^d z=0.
$$
In particular, $y=0$ is a solution (which corresponds to the point $(0:0:1)$ of multiplicity $n_0$. It turns out that if $n_0<d$, it is a dicritical singularity, while if $n_0=d$, a blow-up exhibits a saddle-node with weak separatrix contained in the divisor, so it is not a second type foliation. The same sequence of Cremona transformations as before collapses the origin $(1:0:0)$ and the singularities at infinity in one single singularity with a (generic) divergent separatrix through it.

The rest of the paper will study further properties of foliations with one singular point in low degree and having particular properties. In Sections \ref{unica-singularidad} and \ref{strata} foliations of degrees 2 and 3 will be investigated.

\section{Foliations with a unique singularity with non-zero linear part } \label{unica-singularidad}

In this section, we will study foliations on $\CP$ having only one singularity with non zero linear part. By Proposition \ref{milnor}, the Milnor number of this singularity must be $d^2+d+1$, so this singular point is, either nilpotent or a saddle-node. Let us study independently these two families of foliations.

\subsection{Nilpotent case}
Assume first that the singularity is of nilpotent type. It is well known that has locally a Takens normal form $d(y^2+z^n)+ z^p V(z)dy$, where $V(z)$ is a unit. Nilpotent singularities behave differently according to the following three cases:
\begin{enumerate}
    \item $2p>n$. \textit{Generalized cusp}. It is a generalized curve. It may have either two irreducible separatrices ($n$ even) or one ($n$ odd).
    \item $2p=n$. \textit{Generalized saddle}. Several cases may appear, according to several arithmetic conditions. The foliation may be a generalized curve, dicritical, or a Poincaré-Dulac singularity may appear after blow-ups (in which case it will not be a second type foliation).
    \item $2p<n$. \textit{Generalized saddle-node}. After $p$ blow-ups, two singularities appear away from the corners: one is resonant hyperbolic and another one is a saddle-node. It is then a second type foliation with one or two smooth analytic separatrices, tangent at the origin.
\end{enumerate}

A foliation on $\CP$ with an analytic singularity of nilpotent type is always a generalized saddle-node (\cite{alcantara}). In fact, in appropriate local coordinates, around the singularity the  foliation is formally equivalent to
\begin{equation*}
z^{d^2+d+1}U(z) \frac{\partial}{\partial y}+\left( y+\sum_{i=m}^{d^2+d} a_i z^i \right)\frac{\partial}{\partial z},
\end{equation*}
where $1<m\leq d$, $a_i \in \C$, $a_m\neq 0$, $U(0)\neq 0$. There are, consequently, two separatrices, through the singularities, having the following expansions in these coordinates:

\begin{align*}
\mathcal{C}_1:\ & y+\frac{U(0)}{(m-d^2-d-2)a_m} z^{d^2+d+2-m}+...\\
\mathcal{C}_2:\ & y+a_m z^m+...
\end{align*}

Note that $\mathcal{C}_2$ is the separatrix going through the saddle-node, so it is maybe formal.

The reduction of singularities of this foliation is as follows: after blowing-up points $m$ times, a chain of divisors $E_1,E_2,\ldots , E_m$ is obtained. Over the last one, three singularities appear: $p_0= E_m \cap E_{m-1}$, $p_1$ a resonant singularity with eigenvalues $-1$ and $m$, and $p_2$ a saddle-node. 

Recall the following result concerning the behavior of Milnor numbers after blowing-ups, which can be read in \cite{mattei-moussu} (with a mistake in the statement) and which goes back to van den Essen \cite{vandenessen}: take a singularity $p$ of a holomorphic foliation on $\C^{2}$, with multiplicity $\nu$ and Milnor number $\mu$. If the blow-up of $p$ is non dicritical, then 
$$
\mu= \nu^{2}-\nu -1 + \sum_{c\in E_p} \mu_c,
$$
where $E_p$ is the divisor obtained after blowing up $p$ and $\mu_c$ are the Milnor numbers at every point $c\in E_p$.

Applying this result to the sequence of blowing 
-ups needed to desingularize a generalized saddle-node, if $\mu$ is the Milnor number of the singularity before blow-ups, the saddle-node appearing at the end of the desingularization process has Milnor number $d^2+2$ (as resonant singularities have Milnor number 1). This means that, according to \cite{martinet-ramis-sn}, through it, there is a (possibly) divergent separatrix whose equation has Gevrey order $\dfrac{1}{d^2+1}$ (and it is in fact summable).

Regarding indices, the following corollary collects some results about them:
\begin{cor} \label{indices-nilpotentes}  Let $F$ be a foliation on $\CP$ of degree $d$ with a unique nilpotent singularity. Suppose that its singular local scheme is:

 $$I=\Big \langle y+\sum_{i=m}^{d^2+d} a_iz^i, z^{d^2+d+1} \Big\rangle \subset \mathcal{O}_{\C^2,0},$$ 
 
 \noindent where $a_m \neq 0$. Then $F$ has two irreducible separatrices $\mathcal{C}_1=\mathbb{V}(f_1)$ and $\mathcal{C}_2=\mathbb{V}(f_2)$, the second one maybe formal, and they satisfy:

\begin{enumerate}
\item $GSV_p(F,\mathcal{C}_1)=m$
\item $GSV_p(F,\mathcal{C}_2)=d^2+d+2-m$
\item $CS_p(F,\mathcal{C}_1)=0$
\item $CS_p(F,\mathcal{C}_2)=-d^2+2d+2m$
\item $(\mathcal{C}_1,\mathcal{C}_2)=m$
\item $f_1f_2 \in I\widehat{\mathcal{O}}_{\C^2,0}$ ($\tau_{f_1f_2}=2m-1$, where $\tau_f$ is the Tjurina number of the curve $f$ at $p$).
\item $f_i \notin I$ for $i=1,2$.
\end{enumerate}
\end{cor}

\begin{proof} (1) and (2) are easily computed from Takens normal form. Indeed, if $y=h_i (z)$ is the equation of each of the separatrices ($i=1,2$), we have that $h_1 (z)$ has order $n-p$ and $h_2 (z)$ has order $p$. As they are smooth, by \eqref{GSV-smooth} it is enough to compute
$$
\dim_{\mathbb C} \dfrac{\mathcal{O}_{\C^2,(0,0)}} {(2y+z^p V(x),z^{n-1},y-h_i (z))} = \min \{ \ord (2h(z)+z^p V(z), n-1\},
$$
which is $p$ in the case of $\mathcal{C}_1$ and $n-p$ for $\mathcal{C}_2$ (in this Corollary, $m=p$ and $d^2+d+2=n$). Also, (3) is easily computed as it is the index of a separatrix through a hyperbolic singularity, and eigenvalues are straighforward to be computed. Both separatrices agree, in Takens normal form, up to the term $z^p$ ($z^m$ with the notation of this Corollary), so (5) follows. Using a formal version of Baum-Bott's formula we have that
\begin{align*}
(d+2)^2&=BB_p(F)=CS_p(F,\mathcal{C}_1)+CS_p(F,\mathcal{C}_2)+\\
&2(\mathcal{C}_1,\mathcal{C}_2)+2GSV_p(F,\mathcal{C}_1)+2GSV_p(F,\mathcal{C}_2)-4(\mathcal{C}_1,\mathcal{C}_2),  
\end{align*}
so we obtain $CS_p(F,\mathcal{C}_2)=-d^2+2d+2m$. Both separatrices being smooth, from (5) we deduce that   $\tau_{f_1f_2}=2m-1$. Then 

\begin{align*}
 GSV_p(F,C)&=\dim \frac{\widehat{\mathcal{O}}_{\C^2,0}}{\langle I, f_1 \cdot f_2 \rangle}-\tau_{f_1f_2}\\ 
 &=d^2+d+2-2m,
\end{align*}

\noindent therefore $\dim \frac{\widehat{\mathcal{O}}_{\C^2,0}}{\langle I, f_1 \cdot f_2 \rangle}=d^2+d+1=\dim \frac{\widehat{\mathcal{O}}_{\C^2,0}}{I}$, and $f_1f_2 \in I\widehat{\mathcal{O}}_{\C^2,0}$ (6). Since each separatrix $f_j$ is smooth, then

$$GSV_p(F,\mathcal{C}_j)=\dim \frac{\widehat{\mathcal{O}}_{\C^2,0}}{\langle I, f_j \rangle}<d^2+d+1$$

\noindent for $j=1,2$, therefore $f_j \notin I$. As the singularity is of second type, this implies that it is (locally and formally) quasi-homogeneous \cite{mattei,Mattei-Salem}.
\end{proof}

The following is an explicit  example of a family of foliations of even degree, with only one singularity of nilpotent type:
\begin{theorem} \label{theorem2} Let $d$ be an odd number greater than $1$.  The foliations on $\CP$ given by the vector field:
 \begin{equation*}
\XX= \alpha y^d \frac{\partial}{\partial x}+(\beta x^{\frac{d-1}{2}}yz^{\frac{d-1}{2}}-\beta^2z^d) \frac{\partial}{\partial y}+(x^{d-1}y-\beta x^{\frac{d-1}{2}}z^{\frac{d+1}{2}}) \frac{\partial}{\partial z},
 \end{equation*}

\noindent  where $\alpha, \beta \in \C^*$ have degree $d$ and a nilpotent singularity in $(1:0:0)$ with Milnor number $d^2+d+1$. 

\end{theorem}

For the family of Theorem \ref{theorem2}, $m= \dfrac{d+1}{2}$ and the formal separatrix turns out to be of the form $y=h (z)$, where $h(z)=z^{\frac{d+1}{2}} \tilde{h} (z)$, with $\tilde{h}(0)= \beta$. In fact, it is not hard to check that $\tilde{h} (z) = W(t)$, where $z^{\frac{d^2+1}{2}}=t$ and $W(t)$ satisfies the differential equation
$$
(W(t)-\beta -\alpha tW(t)^{d}) \cdot \left[ \frac{d^2+1}{2} tW'(t) + \frac{d-1}{2} W(t) \right] + (W(t) - \beta)^{2}=0.
$$

From previous considerations, the formal power series defining $W(t)$ must be of Gevrey order $\dfrac{d^2+1}{2}\cdot \dfrac{1}{d^2+1}=\dfrac{1}{2}$. Numerical experiments performed with the first 50 terms of the expression of $W(t)$ tend to support the hypotheses that $W(t)$ is, in fact, divergent for every value of $\alpha$, $\beta \neq 0$. It would be good, for $d=3$, to find an explicit example of a foliation of this type with a non-convergent separatrix, but up to now, we don't know it. It would be the first example of a second type foliation with only one singular point, and a divergent separatrix. It will be the subject of further investigations.

Remark that these foliations have no algebraic separatrices. For the sake of completeness, let us provide here a proof. Denote $\mathcal{C}$ a possible invariant algebraic curve of degree $t$. 
If $\mathcal{C}$ is not smooth ($t\geq 3$), we have by \eqref{CS}, that
$$
t^2= (-d^2+2d+2m)+2m=-d^2+2d+4m\leq -d^2+6d,
$$
so $t^2+d^2\leq 6d$. The only possible solution is $t=3$ and $d=3$. But a cubic cannot have a singular point with two branches of contact number $m\geq 2$. So, $\mathcal{C}$ must be smooth, and the only possibility is that it coincides with the separatrix which goes to the saddle-node appearing after the reduction of singularities (i.e. with $\mathcal{C}_2$ of Corollary \ref{indices-nilpotentes}). We have, by \eqref{GSV},
$$
d^2+d+2-m=t(d+2-t)\leq \left( \frac{d+2}{2} \right)^{2}.
$$
As $m<\dfrac{d^2+d+2}{2}$, then $\dfrac{d^2+d+2}{2}<\left( \dfrac{d+2}{2} \right)^2$, so $d^2-2d<0$, and $d=1$, which is not the case.

\subsection{Saddle-node case}

Let us comment now on the saddle-node case. A saddle-node singularity with Milnor number $k+1$ has a Dulac formal normal form 
$$
\omega = z^{k+1}dy - y (1+\lambda z^{k})dz
$$
(see \cite{mattei-moussu} for a proof). It has an analytic separatrix $\mathcal{C}_1$ ($z=0$ in previous expression) and a maybe formal one, $\mathcal{C}_2$ ($y=0$). It is straightforward to compute the indices in this case, and in fact we have $CS_p(F,\mathcal{C}_1)=0$, $CS_p(F,\mathcal{C}_2)=\lambda$, $GSV_p(F,\mathcal{C}_1)=1$, $GSV_p(F,\mathcal{C}_2)=k+1$.

Assume that $F$ is a degree $d$ foliation with only one singularity of saddle-node type, and consider the previous normal form. By Baum-Bott's formula,
\begin{align*}
BB_p(F)&=(d+2)^2=CS_p(F)+2 GSV_p(F)\\
&=CS_p(F,C_1)+CS_p(F,C_2)+2(C_1 \cdot C_2)_p+\\
&2 \big(GSV_p(F,C_1)+GSV_p(F,C_2)-2(C_1 \cdot C_2)_p \big)\\
&=\lambda +2 +2(1+d^2+d+1-2)=\lambda+2(d^2+d+1).
\end{align*}
So, $\lambda=-d^2+2d+2$. 
 
As in the nilpotent case, no algebraic separatrices can appear (see for instance \cite{alcantara-pantaleon} for a simple proof). As before, let us provide alternative proof of this fact. If the degree $t$ separatrix $\mathcal{C}$ is smooth, by Camacho-Sad index Theorem \eqref{CS} we must have $t^2=0$ or $t^2= -d^2+2d+2$, which is impossible. In the non smooth case, this formula gives the equality
$$
t^2= -d^2+2d+2 +2\cdot 1 = -d^2+2d+4,
$$
so $(d-1)^2=5-t^2$. The solutions $(d,t)$ of this equation  are $(3,1)$ and $(2,2)$. In the first case, $\mathcal{C}$ is a straight line, which is excluded. In the second case, $\mathcal{C}$ is a singular conic, so there would be invariant lines tangent to the foliation.

\begin{remark}
As in the nilpotent case, it is a problem to find an explicit example of a projective foliation with only a singular point, being a saddle-node, and with a divergent weak separatrix. To approach this result, we shall study in the rest of the paper foliations on $\CP$ of low degree.
\end{remark}

\begin{example} Up to change of coordinates, there are only four foliations on $\CP$ of degree 2 with a unique singular point (see \cite{cerveau}). For one of these, the multiplicity of the singular point is 1. There is another one with a rational first integral of degree $3$, this means that the singular point is dicritical. The remaining two are defined by the following vector fields:

\begin{align*}
\XX_2&=-y^2 \frac{\partial}{\partial x}+(zy-z^2) \frac{\partial}{\partial y}+z^2\frac{\partial}{\partial z}\\
\XX_3&=(y^2+z^2) \frac{\partial}{\partial x}+yz \frac{\partial}{\partial y}
\end{align*}
Regarding $\XX_2$, blowing up twice 
the local 1-form $\omega_2=z^2dz+(z+y^2)(ydz-zdy)$ defining the foliation, we obtain  the 1-form:
$$(-y^2+(2z-1)y)dz+z^2dy,$$

\noindent where we have two components of the exceptional divisor and one saddle-node singularity in the intersection of these components, the weak separatrix is tangent to this divisor. Therefore, it is not a second type foliation.
\\

For $\XX_3$ we have the local 1-form:

$$\omega_3=yzdz+(y^2+z^2)(zdy-ydz).$$

\noindent We must do 4 blow-ups to obtain a saddle-node singularity in a corner. Then $\XX_3$ is neither a second type.
\end{example}

\begin{prop} The weak separatrix of a foliation $F$ on $\CP$ of degree $d > 1$ with a unique singular point of saddle-node type is contained in the local formal singular scheme of $F$ which has the form:

$$I \widehat{\mathcal{O}}_0=\Big\langle y+\sum_{i=2}^{d^2+d} a_iz^i, z^{d^2+d+1} \Big\rangle \subset \widehat{\mathcal{O}}_0$$
\end{prop}

\begin{proof} Let $f(y,z)$ be the formal local equation of the weak separatrix $\mathcal{C}$ of $F$ at the singularity $p$. We know that:

$$GSV_p(F,\mathcal{C})=\dim_{\C} \frac{\widehat{\mathcal{O}}_0}{\langle I,f \rangle}- \tau_f=d^2+d+1= \dim_{\C} \frac{\widehat{\mathcal{O}}_0}{I},$$

\noindent as $f$ is smooth (and consequently $\tau_f=0$). Therefore, $f\in I$. 
\end{proof}

\section{Generalized curves and second type foliations}

This section is focused on studying and classifying generalized curves and second type foliations on $\CP$ with a unique singular point. We begin with the following lemma.

\begin{lem} \label{logarithmic} Let $F$ be a foliation on $\CP$ of degree $d$ with a unique non-dicritical singular point. If $F$ has an algebraic invariant curve $\mathcal{C}$ which contains all the separatrices, then the following statements are equivalent:
\begin{enumerate}
\item $F$ is logarithmic,
\item $F$ is a generalized curve,
\item $F$ is a foliation of the second type.
\end{enumerate}
\end{lem}

\begin{proof} Let $F$ be a  foliation on $\CP$ of degree $d$ with a unique non-dicritical singular point $p$ and with an algebraic invariant curve $\mathcal{C}$ of degree $t$. 
\\

Remember that a foliation on $\CP$ is logarithmic if there are homogeneous, irreducible polynomials $f_1,\ldots ,f_m \in \C[x,y,z]$, of degrees $d_1,\ldots ,d_m$, respectively, such that a 1-form describing the foliation is:

\begin{equation*}
\Omega=f_1 \cdot \ldots\cdot f_m \sum_{i=1}^m \lambda_i \frac{df_i}{f_i},
\end{equation*}

\noindent where $\lambda_1,\ldots, \lambda_m$ are non-zero complex numbers such that $\sum_{i=1}^m \lambda_i d_i=0$. Note that the foliation has degree $\sum_{i=1}^m d_i-2$. A good reference for studying logarithmic foliations on $\CP$ where the poles are not necessarily in a normal crossing divisor is  \cite{cerveau-jsing}.  
\\

A  logarithmic foliation on $\CP$ is always a generalized curve \cite[Sec. 6]{fernandez-mol}, and a generalized curve foliation is of course a foliation of the second type.
\\

Now assume that $F$ is a foliation of the second type. 
From \eqref{CS}, \eqref{GSV}, \eqref{BB} and \eqref{formula-indices} we have 
$(d+2)^2=t^2+2((d+2)t-t^2)$, as $\mathcal{C}$ contains all the separatrices through $p$. We conclude that $t=d+2$. By  \cite[Thm. 4]{cerveau-jsing},  $F$ is a logarithmic foliation.
\end{proof}

\subsection{Foliations with an algebraic leaf}

To determine if a foliation on $\CP$ has algebraic leaves is a very difficult task. Lemma \ref{logarithmic} can be used to prove this for the case of foliations with a unique singular point. For example, if the foliation is of  second type but it is not a generalized curve, then the foliation cannot have algebraic leaves containing all separatrices.
\\

Once we know that there is some algebraic leaf it is also difficult to find the polynomial equations defining it. It is for this reason that results such as the following are very useful.

\begin{prop} Let $d \geq 2$, and $F \in \F_d$ be a foliation with a unique non-dicritical singular point $p$. If $f(x,y,z) \in \C[x,y,z]$ is a reduced equation of an algebraic curve $\mathcal{C}$ which is invariant by $F$ and if $f$ is in the singular scheme of $F$, then $F$ is logarithmic and the germ of $f$  at $p$ is quasi-homogeneous.
\end{prop}

\begin{proof} Assume that $f$ is in the singular scheme of $F$. By \cite[Thm. 3.2]{campillo-olivares}, we know that in the singular scheme of $F$ there are no curves of degree less than or equal to $d$, therefore $d+1 \leq t \leq d+2$.
\\

We  fix the singular point in $p=(1:0:0)$ and take the local ideal 
$I \subset \C[y,z]$. Denote $f_{\ast}=f(1,y,z)$. Then $\dim \frac{\C[y,z]}{\langle I,f_{\ast} \rangle}=d^2+d+1$, since $f\in I$. Therefore:

$$GSV_p(F,f)=\dim \frac{\C[y,z]}{\langle I,f_{\ast} \rangle}-\tau_f=d^2+d+1-\tau_f=t(d+2)-t^2.$$

\noindent If $t=d+1$ this index is $d+1$ and
$\tau_f=d^2$; but this is impossible because, by  \cite[Thm. 1.1]{ploski}, $\tau_f \leq \mu_f \leq (t-1)^2-\left[ \dfrac{t}{2}\right] =d^2-\left[ \dfrac{d-1}{2}\right]$.
\\

If $t=d+2$, we have that $F$ is a logarithmic foliation \cite[Thm. 1.4]{cerveau-jsing} and $\mathcal{C}$ contains all the separatrices through $p$. 

From Lemma \ref{logarithmic}, we have that the foliation $F$ is a generalized curve then, by  \cite[Thm. A]{mattei}, $f_{\ast}$ is quasi-homogeneous because it is in the local singular scheme of the foliation.
\end{proof}

In fact, we have some evidence that indicates that this type of foliations (logarithmic, non dicritical, and quasi-homogeneous, with a unique singular point) do not exist, but for now, we consider it as a conjecture.
\\

The study of foliations where the unique singular point is dicritical and we have algebraic leaves will be developed in the following section.

\section{Foliations with a rational first integral} \label{rational-first-integral}

The well-known Poincar\'e problem for foliations on $\CP$ states the following:  given a foliation on $\CP$ of degree $d$ with a rational first integral of degree $s$ then, can we bound $s$ as function of $d$? 
The answer in general is no, as the following classic example shows. Consider the foliation on $\CP$ of degree 1 given by the 1-form
\begin{equation*}
pyzdx+qxzdy-(p+q)yxdz,
\end{equation*}
\noindent with $p$ and $q$ positive integers. Then the pencil  $\{\alpha x^py^q - \beta z^{p+q}\}_{(\alpha:\beta) \in \mathbb{CP}^1}$ of degree $p+q$ defines its rational first integral. We can choose $s=p+q$ arbitrarily large but the foliation always has degree $1$.
\\

Let $F$ be a foliation on $\CP$ of degree $d$ with a unique singular point at $p=(1:0:0)$ and with a rational first integral given globally by the pencil:

\begin{equation*}
\big\{ \alpha G(x,y,z)-\beta H(x,y,z) \big\}_{(\alpha:\beta ) \in \mathbb{CP}^1},
\end{equation*}

\noindent where $G$ and $H$ are algebraic plane curves of degree $s$ with no common factors (without loss of generality we can suppose that $G$ is irreducible). Then there exists a finite set $\{(\alpha_1: \beta_1),\ldots ,(\alpha_k: \beta_k)\} \subset \mathbb{CP}^1$ such that the corresponding fibers are reducible or not reduced (see \cite{jouanolou}). Suppose:

\begin{equation*}
\alpha_i G(x,y,z)-\beta_i H(x,y,z)=\prod_{j_i} U_{i,j_i}^{n_{i,j_i}},
\end{equation*}

\noindent where $U_{i,j_i}$ are irreducible plane curves of degrees $d_{i,j_i}$ and $n_{i, j_i} \in \N$. If we denote $R=\prod_{i=1}^k \prod_{j_i} U_{i,j_i}^{n_{i,j_i}-1}$ and we consider the 1-form $\Omega$ defining the foliation $F$, then it satisfies:

\begin{equation} \label{eq}
GdH-HdG=R \Omega,
\end{equation}

\noindent which implies that if $R$ has degree $N$ then $2s-2=N+d$. Observe that $R$ is a union of leaves of $F$, and for $i=1,\ldots,k$, we have that the degree of $\prod_{j_i} U_{i,j_i}^{n_{i,j_i}-1}$ is $\sum_{j_i} d_{i,j_i}(n_{i,j_i}-1) < s$.
\\

In this section, we study some foliations with a unique singular point and with a rational first integral of degree $s$. Due to the above, we have that
$s \geq \frac{d+2}{2}$. We are going to give some examples where this bound is reached and some others that represent a counterexample to the Poincar\'e problem.

\begin{example} It is easy to see that the unique foliation on $\CP$ of degree $1$ with a unique singular point is, up to change of coordinates by the special linear group $SL(3,\C)$, the following:
$$z^2dx-yzdy+(y^2-xz)dz,$$
\noindent and this foliation has a rational first integral of degree $2$, given by the pencil of conics: 
$$\{\alpha(2xz+y^2)-\beta z^2\}_{(\alpha:\beta)  \in \mathbb{CP}^1}.$$
\end{example}
The following theorem gives, for every even number $d$, examples of foliations on $\CP$ of degree $d$ with a unique singular point and with a rational first integral of degree $\frac{d+2}{2}$.

\begin{theorem} \label{final} Let $n \in \N_{\geq 2}$. Consider the pencil of degree $n+1$:
$$\left\{ \alpha y(xy^{n-1}+z^n)-\beta(z(xy^{n-1}+z^n)+y^{n+1})\right\}_{(\alpha:\beta) \in \mathbb{CP}^1},$$
\noindent then the associated foliation on $\CP$ has a unique singular point and it has degree $2n$. This means that the pencil does not have non-reduced fibers.
\end{theorem}

\begin{proof} Let $n \in \N$, we will work in the affine chart $U_0=\{ (1:y:z) \in \CP\}$. Consider the following algebraic curves:
\begin{align*}
f(y,z)&=y(y^{n-1}+z^n)\\
g(y,z)&=z(y^{n-1}+z^n)+y^{n+1}.
\end{align*}
%
\noindent The local form of the foliation is:
\begin{align*}
\omega=&(gf_y-fg_y)dy+(gf_z-fg_z)dz=\\
&(y^{2n-2}z+2y^{n-1}z^{n+1}-y^{2n}-ny^{n+1}z^n+z^{2n+1})dy+\\
&(y(-y^{2n-2}-2y^{n-1}z^n+ny^{n+1}z^{n-1}-z^{2n}))dz,
\end{align*}
\noindent and a simple computation shows that the multiplicity of the singular point $p=(1:0:0)$ is $2n-1$ and the Milnor number is $(2n)^2+2n+1$.
Since the foliation has degree $2n$, then $p$ is the unique singular point and the pencil does not have non-reduced fibers.
\end{proof}

\begin{prop} \label{d+1} The following is an irreducible, locally closed subvariety of dimension $d+4$ of $\F_d$:

\begin{equation*}
\mathcal{R}(d,d+1)=SL(3,\C)\big\{ \XX=P(y,z) \frac{\partial}{\partial x}+z^d \frac{\partial}{\partial y}: P(y,z) \in \C^d[y,z], P(y,0) \neq 0 \big\}.
\end{equation*}

\noindent This is in the space of foliations on $\CP$ of degree $d$ with a unique singular point, with a rational first integral of degree $d+1$.

\end{prop}

\begin{proof} If $P(y,z)=\sum_{i=0}^d a_iy^iz^{d-i}  \in \C^d[y,z]$, with $a_d\neq 0$, then it is easy to see that the rational first integral of:

\begin{equation*}
\XX=P(y,z) \frac{\partial}{\partial x}+z^d \frac{\partial}{\partial y},
\end{equation*}

\noindent is given by the pencil of degree $d+1$:

\begin{equation*}
\left\{ \alpha \left( \sum_{i=0}^d  \frac{a_i}{i+1}y^{i+1}z^{d-i}-xz^d\right) - \beta z^{d+1} \right\}_{(\alpha: \beta ) \in \mathbb{CP}^1}.
\end{equation*}

\noindent Finally, by \cite[Thm. 4.1]{alcantara1}, we have that this space is an irreducible, locally closed algebraic variety of dimension $d+4$. 
\end{proof}

We are going to finish this section by building an example of a foliation of degree 5 with a first integral of not bounded degree.

\begin{example}
Let $n,n_1,n_2,n_3 \in \N$ such that $n_1+n_2+n_3=4n$ and each $n_i$ is relatively prime to $n$. Consider the pencil of degree $4n$:

$$\Big\{ \alpha \big(xzy(z-y)+z^4+y^4\big)^n-\beta z^{n_1}y^{n_2}(z-y)^{n_3}\Big\}_{(\alpha:\beta) \in \mathbb{CP}^1}.$$

\noindent The associated local foliation is:

\begin{align*}
  &\omega=\Big(-n_2z^5y+(n_2+n_3)z^4y^2+(n_1+n_3)zy^5-n_1y^6\\
  &+(2n-n_1-n_3)z^3y^2+(-3n+2n_1+n_3)z^2y^3+(n-n_1)zy^4\Big)dz \\
  &+\Big(n_2z^6+(-n_2-n_3)z^5y+(-n_1-n_3)z^2y^4+n_1zy^5+\\
  &(n-n_2)z^4y+(-3n+2n_2+n_3)z^3y^2+(2n-n_2-n_3)z^2y^3\Big)dy.
\end{align*}

\noindent It has degree $5$, a unique singular point in $(0:0:1)$ and its rational first integral has degree $4n$. 
\end{example}

\section{Strata of the space of foliations and its separatrices} \label{strata}
In this section, we shall study other families of foliations of low degree, with only one singular point. We are interested in non-dicritical foliations, and the possible presence of a divergent separatrix. When the degree is $d=3$, in \cite{alcantara-ronzon} unstable foliations are classified: they fall into families that are called $S_6$, $S_9$, $S_{11}$, $S_{12}$ and $S_{15}$ in the aforementioned paper. Let us study them here from the point of view of the separatrices and reduction of singularities, as we are studying in this paper.

\subsection{Stratum $S_6$}  It is the case of a foliation with a rational first integral, already considered in Section \ref{rational-first-integral} in arbitrary degree.

\subsection{Stratum $S_9$} It is the family
$$
z(az^2+ P(y,z))dy+(bz^3-y (az^2+ P(y,z)))dz,
$$
where $P(y,z)= y^3 +a_2 y^2 z +a_1 yz^2 + a_0 z^3$ is homogeneous of degree 3. The condition $b\neq 0$ is enough to guarantee that only one singularity appears. Assuming $a\neq 0$, we reduce the singularities after three blow-ups, which originate an exceptional divisor with components $E_i$, $i=1,2,3$. The singularities at the end are:
\begin{itemize}
\item $p_{1}  = E_1 \cap E_3$.  It is a saddle-node in a corner. 
\item $p_{2} = E_2\cap E_{3}$. Resonant hyperbolic singularity.
\item $p_{3} \in E_2$.  Hyperbolic singularity, with a separatrix ($z=0$) transversal to the divisor.
\item $p_{4}  \in E_3$.  Saddle-node with weak separatrix transversal to the divisor.
\end{itemize}

Let us study the convergence of the separatrix through $p_{4}$. For this, take local coordinates $(u,v)$ around the corner $p_1$: $y=uv^2$, $z=uv^3$.  In these coordinates the foliation is $v^2 du-u (Q(v) u+1-3v)$, assuming without lost of generality that $a=1$, $b=1$, and $Q(v)= 1+a_2 v+ a_1v^2+a_0v^3$. The component $E_3$ is $v=0$, and the singularity $p_{4}$ appears as $u=-1$. This is  Bernoulli type foliation, which can be explicitly integrated: after the transformation $u=\dfrac{1}{w}$ the new equation is $v^2 dw+ (Q(v)+(1-3v)w)dw$, linear. Integrating this equation it turns out that
$$
w(v) = -(1+(a_2+3)v + (2a_2+a_1+6)v^2 + (2a_2+a_1+a_0+8)z^3)
$$
is a (polynomial) equation of the weak separatrix through $p_4$, which is, in fact, algebraic. In the original coordinates of the foliation, this separatrix  lies over the algebraic cusp of equation
$$
z^2 + (y^3+ (a_2+3)y^2z + (2a_2+a_1+6)yz^2 + (2a_2+a_1 + a_0+6) z^3)=0.
$$
Remark that it is not a second type foliation, due to the presence of a saddle-node in the corner $p_1$.

The family $S_9$, as explained in \cite{alcantara-ronzon}, contains another case: the foliations
$$
zP(y,z) dy + (bz^3 +cyz^2 -y P(y,z))dz,
$$
with $P(y,z)$ as before. The conditions in order to have only one singularity are $c\neq 0$ and $bz^3 +cy z^2$ not dividing $P(y,z)$. After three blow-ups, as before, we have three components $E_1$, $E_2$, and $E_3$ of the exceptional divisor, and the following singularities:
\begin{itemize}
\item $p_{1}  = E_1 \cap E_3$,  hyperbolic resonant.
\item $p_{2}  = E_2\cap E_{3}$,  hyperbolic resonant.
\item $p_{3}  \in E_1$,  saddle-node with a convergent separatrix transversal to the divisor.
\item $p_{4}  \in E_2$,  hyperbolic resonant.
\item $p_{5}  \in E_3$,  with eigenvalues $\{ 1,3\}$. It is, either a Poincaré-Dulac singularity or a dicritical one. A closer study shows that it is, indeed, dicritical.
\end{itemize}

As a conclusion, the foliations of the family $S_9$ are not of second type, and no divergent separatrix appears.

\subsection{Stratum $S_{11}$} It is a family described in local coordinates by the 1-form
$$
\omega = z(a_1 yz+a_0 z^2 + P_3 (y,z))dy + (z^2+Q_2 (y,z) - y (a_1 yz+a_0 z^2 + P_3 (y,z)))dz 
$$
where $Q_2 (y,z)$, $P_3 (y,z)$ are homogeneous polynomials of degrees 2, 3, respectively, and $P_3 (y,0)\neq 0$. The origin is the unique singularity  provided that $P_3 (y,z)= Q_2 (y,z) (a_1 y + a_0 z) + cz^3$, with $c\neq 0$. 

Reduction of singularities is as follows: blow up the origin, obtaining a component $E_1$ of the exceptional divisor, with only one singular point on it.  Blow it up and a second component $E_2$ appears with four singularities on it:
\begin{itemize}
\item $p_0=E_1\cap E_2$, simple, resonant.
\item $p_1$, resonant.
\item $p_2$ is a saddle-node with a weak separatrix transversal to the divisor.
\item $p_3$ with eigenvalues $\lambda_1$, $\lambda_2$, $\lambda_1/\lambda_2 =2$. It is, either a Poincar\'{e}-Dulac singularity, or a dicritical singular point.
\end{itemize}

So, foliations in the family $S_{11}$ are never second type foliations. Nevertheless, let us show in a simplified case how it can be proven that the weak separatrix through $p_2$ actually converges. Indeed, assume that $a_1=1$, $a_0=0$, $Q_2 (y,z)=y^2$, $c=1$. The two blow-ups can be expressed, in appropriate charts, as:
$$
 \begin{cases} y & =y_1 \\ z & = y_1z_1 \end{cases}  \hspace{0.5cm} \text{ and} \qquad  \begin{cases} y_1 & =y_2z_2 \\ z_1 & = z_2 \end{cases} .
$$
Renaming $y_2$, $z_2$ as $y$, $z$, the foliation around $p_0$ is
$$
(1+y) zdy - (y(y^2-y-2) + y^3z^3)dz,
$$
where $E_1 \equiv (y=0)$, $E_2 \equiv (z=0)$. $p_2$ is the point with coordinates $(-1,0)$ and $p_3$ is $(2,0)$.

This differential equation has an implicit solution which can be written as\footnote{Computations done with Maple 2022, licensed for Valladolid University.}
$$
G(y,z)= \frac{aAi \left( N\dfrac{y-2}{yz^2} \right) + bz Ai ' \left( N\dfrac{y-2}{yz^2} \right)}{a Bi \left( N\dfrac{y-2}{yz^2} \right) + b z Bi' \left( N\dfrac{y-2}{yz^2} \right)} ,
$$
where $Ai (x)$, $Bi (x)$ is a basis of (entire functions) solutions of the Airy differential equation $y'' (x) - x y(x)=0$. We will take $Ai (x)$, $Bi (x)$ as the usual basis (see \cite{Olver} for details). In previous expression, $N= \dfrac{1}{2} \left(  \dfrac{3}{2} \right)^{1/3}$, $a= - \left( \dfrac{3}{2} \right)^{2/3}\cdot b$. 

$G(y,z)$ is a first integral for $\omega$ defined if $z\neq 0$, $y\neq 0$. In particular, it is defined over a neighborhood of $p_2$ minus the divisor $z=0$. Let us comment on its behavior when approaching $z=0$. The asymptotics of $Ai (x)$, $Bi (x)$ have been well studied when $x$ tends to infinity, which corresponds to $z$ tending to $0$ in the expression for $G(y,z)$. With notations as in \cite{Olver}, $G(y,z)$ is asymptotic to

\footnotesize
\begin{equation*}
\frac{e^{-2 \xi (y,z)}}{2} \cdot 
 \frac{aN^{-1/4} \eta (y)^{-1/4} \sum_s u_s \eta (y)^{-3s/2} (-2\sqrt{3} z^3)^{s} - bN^{1/4} \eta(y)^{1/4} \sum_s v_s \eta (y)^{-3s/2} (-2 \sqrt{3} z^3)^{s}}{aN^{-1/4} \eta (y)^{-1/4} \sum_s u_s \eta (y)^{-3s/2} (2\sqrt{3} z^3)^{s} + bN^{1/4} \eta(y)^{1/4} \sum_s v_s \eta (y)^{-3s/2} (2 \sqrt{3} z^3)^{s}},
\end{equation*} \normalsize 
where we have denoted $\eta (y)=\dfrac{y-2}{y}$, $u_s= \dfrac{(2s+1)(2s+3) \cdots (6s-1)}{2^{3s} 3^{3s} s!}$, and $v_s= - \dfrac{6s+1}{6s-1} u_s$, for $s\geq 1$, and $\xi (y,z)= \dfrac{1}{2\sqrt{3}} \eta (y)^{3/2} \dfrac{1}{z^3}$.

In a neighbourhood of $y=-1$, the expression concerning the asymptotics of $G(y,z)$ may be written as
$$
e^{-\xi (y,z)} \left[ \sum_{s\geq 0} \alpha_s \eta (y)^{-3s /2} z^{3s} + \eta (y)^{-1/2} \sum_{s\geq 0} \beta_s \eta(y)^{-3s/2} z^{3s} \right],
$$
for certain constants $\alpha_s$, $\beta_s$. In a neighborhood of $y_0= -1 + \varepsilon$ ($\varepsilon$ small), $z=0$, the foliation is regular, so this first integral must be convergent in $z$. There are positive constants $M$, $R$ such that
$$
\norm{\alpha_s } \cdot \norm{ \frac{3-\varepsilon}{1-\varepsilon} }^{-3s/2} \leq \frac{M}{R^{3s}},\qquad \norm{\beta_s } \cdot \norm{ \frac{3-\varepsilon}{1-\varepsilon} }^{-3s/2} \leq \frac{M}{R^{3s}}.
$$
From these bounds, convergence is also guaranteed in a neighborhood of $y_0=-1$, $z=0$ ($\varepsilon=0$). As this is an analytic expression around the singularity $p_2$, it turns out that the weak separatrix, which is a level for $F(y,z)$, actually converges. This argument may be carried out in the general case, but it is technically much more complicated and we shall not attempt to do it here.

\subsection{Stratum $S_{12}$} It consists on the family
\begin{multline*}
z((\alpha_1 y - \beta_1 z) (\alpha_2 y-\beta_2 z) + P_3 (y,z))dy \\ + (\alpha (\alpha_1 y - \beta_1 z)^2 (\alpha_2 y-\beta_2 z) y- y ((\alpha_1 y - \beta_1 z) (\alpha_2 y-\beta_2 z) + P_3 (y,z)) )dz,
\end{multline*}
where $P_3 (y,z)$ is a homogeneous polynomial of degree 3, $P_3 (y,0)=y^3$, and in order to have the origin as the unique isolated singularity, we must impose that $\alpha_1 \alpha_2 (\alpha \alpha_1-1)=0$, $\beta_1\beta_2 \neq 0$, and $P_3 (\beta_1,\alpha_1) P_3 (\beta_2,\alpha_2) \neq 0$.

The reduction of the singularities depends largely on the precise values of the parameters, but some general comments may be done. After blowing up the origin, a component $E_1$ appears with two singular points on it: $p_1$ and $p_2$. Blowing up $p_1$, a new divisor $E_2$ is obtained, with three singular points: the corner $p_3=E_1\cap E_2$, and two other points $p_4$, $p_5$. One of them (say, $p_4$) is reduced, resonant, with Camacho-Sad index $-1$. The other points ($p_3$ and $p_5$) have non-zero linear parts, with opposite Camacho-Sad indices. 	If these indices are not rational, they are reduced singularities. If they are rational, one of them is reduced, and the other needs still to be desingularized, arriving either to a dicritical situation or to a Poincar\'{e}-Dulac type singularity. In any case, it is not a second type foliation, and no divergent separatrices appear.

A similar situation is obtained during the reduction of singularities of $p_2$, and we shall not detail it here.
\subsection{Stratum $S_{15}$} It is, in local coordinates, the family 
$$
z(z+a_1 yz+a_0 z^2+P(y,z))dy + (bz^3-y(z+a_1 yz+a_0 z^2+P(y,z))dz,
$$
where $b\neq 0$, and $P$ is a homogeneous polynomial of degree 3, with nonzero term in $y^3$, which we can assume that it is 1. It is a dicritical foliation, the dicritical component of the divisor appears in the first blow-up.

%
%
%
%
%
%
%
%
%
%
%

\subsection*{Acknowledgments} The first author wishes to thank the hospitality of the Universidad de Valladolid during the stay in which much of this work was carried out.
The first author also wants to thank Dominique Cerveau for helpful conversations which helped to improve this work.

The second author wants to thank the Universidad de Guanajuato for the hospitality and support. He also wants to thank the city of Cuévano for the amazing ambiance while preparing this work. And last, but not least, the second author wishes to thank warmly the first author for being there, and for so many wonderful moments spent together, both physically and online. Claudia, I will never be grateful enough.

 Both authors want to thank the referees for their useful comments.

\end{document}